\documentclass[11pt]{article}%
\usepackage{amsmath}
\usepackage{amsfonts}
\usepackage{amssymb}
\usepackage{graphicx}%
\newtheorem{theorem}{Theorem}

\newtheorem{corollary}[theorem]{Corollary}

\newtheorem{definition}[theorem]{Definition}
\newtheorem{example}[theorem]{Example}

\newtheorem{lemma}[theorem]{Lemma}

\newtheorem{proposition}[theorem]{Proposition}
\newtheorem{remark}[theorem]{Remark}

\newenvironment{proof}[1][Proof]{\noindent\textbf{#1.} }{\ \rule{0.5em}{0.5em}}
\begin{document}

\title{Some Singular Vector-valued Jack and Macdonald Polynomials}
\author{Charles F. Dunkl\thanks{Dept. of Mathematics, University of Virginia,
Charlottesville VA 22904-4137; email: cfd5z@virginia.edu}}
\maketitle

\begin{abstract}
For each partition $\tau$ of $N$ there are irreducible modules of the
symmetric groups $\mathcal{S}_{N}$ or the corresponding Hecke algebra
$\mathcal{H}_{N}\left(  t\right)  $ whose bases consist of reverse standard
Young tableaux of shape $\tau$. There are associated spaces of nonsymmetric
Jack and Macdonald polynomials taking values in these modules,
respectively.The Jack polynomials are a special case of those constructed by
Griffeth for the infinite family $G\left(  n,p,N\right)  $ of complex
reflection groups. The Macdonald polynomials were constructed by Luque and the
author. For both the group $\mathcal{S}_{N}$ and the Hecke algebra
$\mathcal{H}_{N}\left(  t\right)  $ there is a commutative set of Dunkl
operators. The Jack and the Macdonald polynomials are parametrized by $\kappa$
and $\left(  q,t\right)  $ respectively. For certain values of the parameters
(called singular values) there are polynomials annihilated by each Dunkl
operator; these are called singular polynomials. This paper analyzes the
singular polynomials whose leading term is $x_{1}^{m}\otimes S$, where $S$ is
an arbitrary reverse standard Young tableau of shape $\tau$. The singular
values depend on properties of the edge of the Ferrers diagram of $\tau$.

\end{abstract}

\section{Introduction}

For each partition $\tau$ of $N$ there are irreducible modules of the
symmetric groups $\mathcal{S}_{N}$ and the corresponding Hecke algebra
$\mathcal{H}_{N}\left(  t\right)  $, whose bases consist of reverse standard
Young tableaux of shape $\tau$. There are associated spaces of nonsymmetric
Jack and Macdonald polynomials taking values in these modules, respectively.
(In what follows the polynomials are always of the nonsymmetric type.) The
Jack polynomials are a special case of those constructed by Griffeth
\cite{G2010} for the infinite family $G\left(  n,p,N\right)  $ of complex
reflection groups. The Macdonald polynomials were constructed by Luque and the
author \cite{DL2012}. The polynomials are the simultaneous eigenfunctions of
the Cherednik operators, which form a commutative set. For both the group
$\mathcal{S}_{N}$ and the Hecke algebra $\mathcal{H}_{N}\left(  t\right)  $
there is a commutative set of Dunkl operators, which lower the degree of a
homogeneous polynomial by $1$.

The Jack and the Macdonald polynomials are parametrized by $\kappa$ and
$\left(  q,t\right)  $ respectively. For certain values of the parameters
(called singular values) there are polynomials annihilated by each Dunkl
operator; these are called singular polynomials. The structure of the singular
polynomials for the trivial module corresponding to the partition $\left(
N\right)  $, that is the ordinary scalar polynomials, is more or less well
understood by now. For the modules of dimension $\geq2$ the singular
polynomials are mostly a mystery. In \cite{Dunkl2018} and \cite{Dunkl2019} we
constructed special singular polynomials which correspond to the minimum
parameter values. To be specific denote the longest hook-length in the Ferrers
diagram of $\tau$ by $h_{\tau}$ then any other singular value $\kappa$
satisfies $\left\vert \kappa\right\vert \geq\frac{1}{h}$ and if a pair
$\left(  q,t\right)  $ such that $q^{m}t^{n}=1$ provides a singular polynomial
then $\left\vert \frac{m}{n}\right\vert \geq\frac{1}{h}$. The main topic of
this paper is the determination of all the singular values for which the Jack
or Macdonald polynomials with leading term $x_{1}^{m}\otimes S$ are singular,
where $S$ is an arbitrary reverse standard Young tableau of shape $\tau$. The
singular values depend on properties of the edge of the Ferrers diagram of
$\tau$.

There is a brief outline of the needed aspects of the representation theory of
$\mathcal{S}_{N}$ and $\mathcal{H}_{N}\left(  t\right)  $ in Section
\ref{RepTh}, focussing on the action of the generators on the basis elements.
The important operators on scalar and vector-valued polynomials are defined in
Section \ref{PolyRep}. Subsection \ref{JPols} deals with the Cherednik-Dunkl
and Dunkl operators on the vector-valued polynomials, introduces the Jack
polynomials, and the key formulas for the action of Dunkl operators, in
particular, when specialized to the polynomials with leading term $x_{1}%
^{m}\otimes S$. Subsection \ref{Mpols} contains the analogous results on
Macdonald polynomials. Section \ref{SingPol} combines the previous results
with analyses of the spectral vectors and a combinatorial analysis of the
possible singular values, to prove our main results on Jack and Macdonald
polynomials. Subsection \ref{Isotyp} illustrates the representation-theoretic
aspect of singular polynomials.

\section{Representation Theory\label{RepTh}}

The symmetric group $\mathcal{S}_{N}$ is the group of permutations of
$\left\{  1,2,\ldots,N\right\}  $. The transpositions $w=\left(  i,j\right)
$, defined by $w\left(  i\right)  =j,w\left(  j\right)  =i$ and $w\left(
k\right)  =k$ for $k\neq i,j$ are fundamental tools in this study. The simple
reflections $s_{i}:=\left(  i,i+1\right)  ,1\leq i<N$, generate $\mathcal{S}%
_{N}$ ; and the group is abstractly presented by $\left\{  s_{i}^{2}=1:1\leq
i<N\right\}  $ and the \textit{braid relations:}%
\begin{align*}
s_{i}s_{i+1}s_{i}  &  =s_{i+1}s_{i}s_{i+1},1\leq i\leq N-2,\\
s_{i}s_{j}  &  =s_{j}s_{i},1\leq i<j-1\leq N-2.
\end{align*}

The group algebra $\mathbb{C}\mathcal{S}_{N}$, namely the linear space
$\left\{  \sum_{w\in\mathcal{S}_{N}}c_{w}w\right\}  $, is of dimension $N!$.
The associated Hecke algebra $\mathcal{H}_{N}\left(  t\right)  $ where $t$ is
transcendental (formal parameter) or a complex number not a root of unity, is
the associative algebra generated by $\left\{  T_{1},T_{2},\ldots
,T_{N-1}\right\}  $ subject to the relations%
\begin{gather*}
\left(  T_{i}+1\right)  \left(  T_{i}-t\right)  =0,\\
T_{i}T_{i+1}T_{i}=T_{i+1}T_{i}T_{i+1},1\leq i\leq N-2,\\
T_{i}T_{j}=T_{j}T_{i},1\leq i<j-1\leq N-2.
\end{gather*}
It can be shown that there is a linear isomorphism between $\mathbb{C}%
\mathcal{S}_{N}$ and $\mathcal{H}_{N}\left(  t\right)  $ based on the map
$s_{i}\rightarrow T_{i}$. When $t=1$ they are identical.

The irreducible modules of these algebras correspond to partitions of $N$ and
are constructed in terms of Young tableaux. The descriptions will be given in
terms of the actions of $\left\{  s_{i}\right\}  $ or $\left\{  T_{i}\right\}
$ on the basis elements (see \cite{DJ1986}).

Let $\mathbb{N}_{0}:=\left\{  0,1,2,3,\ldots\right\}  $ and denote the set of
partitions\linebreak\ $\mathbb{N}_{0}^{N,+}:=\left\{  \lambda\in\mathbb{N}%
_{0}^{N}:\lambda_{1}\geq\lambda_{2}\geq\cdots\geq\lambda_{N}\right\}  $. Let
$\tau$ be a partition of $N$ that is $\tau\in\mathbb{N}_{0}^{N,+}$ and
$\left\vert \tau\right\vert =N$. Thus $\tau=\left(  \tau_{1},\tau_{2}%
,\ldots\right)  $ (often the trailing zero entries are dropped when writing
$\tau$). The length of $\tau$ is $\ell\left(  \tau\right)  :=\max\left\{
i:\tau_{i}>0\right\}  $. There is a Ferrers diagram of shape $\tau$ (given the
same label), with boxes at points $\left(  i,j\right)  $ with $1\leq i\leq
\ell\left(  \tau\right)  $ and $1\leq j\leq\tau_{i}$. A \textit{tableau} of
shape $\tau$ is a filling of the boxes with numbers, and a \textit{reverse
standard Young tableau} (RSYT) is a filling with the numbers $\left\{
1,2,\ldots,N\right\}  $ so that the entries decrease in each row and each
column. Denote the set of RSYT's of shape $\tau$ by $\mathcal{Y}\left(
\tau\right)  $ and let $V_{\tau}=\mathrm{span}_{\mathbb{F}}\left\{
S:S\in\mathcal{Y}\left(  \tau\right)  \right\}  $ with orthogonal basis
$\mathcal{Y}\left(  \tau\right)  $, (where $\mathbb{F}$ is some extension
field of $\mathbb{Q}$ containing the parameters $\kappa$ or $q,t$). The
dimension of $V_{\tau}$, that is $\#\mathcal{Y}\left(  \tau\right)  $, is
given by the well-known hook-length formula. For $1\leq i\leq N$ and
$S\in\mathcal{Y}\left(  \tau\right)  $ the entry $i$ is at coordinates
$\left(  \mathrm{\operatorname{row}}\left(  i,S\right)  ,\operatorname{col}%
\left(  i,S\right)  \right)  $ and the \textit{content} of the entry is
$c\left(  i,S\right)  =\operatorname{col}\left(  i,S\right)
-\mathrm{\operatorname{row}}\left(  i,S\right)  $. Each $S\in\mathcal{Y}%
\left(  \tau\right)  $ is uniquely determined by its \textit{content vector}
$\left[  c\left(  i,S\right)  \right]  _{i=1}^{N}$. For example let
$\tau=\left(  4,3\right)  $ and $S=$ $%
\begin{array}
[c]{cccc}%
7 & 6 & 5 & 2\\
4 & 3 & 1 &
\end{array}
$ then the content vector is $\left[  1,3,0,-1,2,1,0\right]  $. There are
representations of $\mathcal{S}_{N}$ and $\mathcal{H}_{N}\left(  t\right)  $
on $V_{\tau}$; each will be denoted by $\tau$. For each $i$ and $S$ (with
$1\leq i<N$ and $S\in\mathcal{Y}\left(  \tau\right)  $) there are four
different possibilities:

1) $\mathrm{\operatorname{row}}\left(  i,S\right)  =\mathrm{\operatorname{row}%
}\left(  i+1,S\right)  $ (implying $\operatorname{col}\left(  i,S\right)
=\operatorname{col}\left(  i+1,S\right)  +1$ and $c\left(  i,S\right)
-c\left(  i+1,S\right)  =1$) then%
\[
S\tau\left(  s_{i}\right)  =S,~S\tau\left(  T_{i}\right)  =tS;
\]

2) $\operatorname{col}\left(  i,S\right)  =\operatorname{col}\left(
i+1,S\right)  $ (implying $\mathrm{\operatorname{row}}\left(  i,S\right)
=\mathrm{\operatorname{row}}\left(  i+1,S\right)  +1$ and $c\left(
i,S\right)  -c\left(  i+1,S\right)  =-1$) then%
\[
S\tau\left(  s_{i}\right)  =-S,~S\tau\left(  T_{i}\right)  =-S;
\]

3) $\mathrm{\operatorname{row}}\left(  i,S\right)  <\mathrm{\operatorname{row}%
}\left(  i+1,S\right)  $ and $\operatorname{col}\left(  i,S\right)
>\operatorname{col}\left(  i+1,S\right)  $. In this case%
\[
c\left(  i,S\right)  -c\left(  i+1,S\right)  =\left(  \operatorname{col}%
\left(  i,S\right)  -\operatorname{col}\left(  i+1,S\right)  \right)  +\left(
\mathrm{\operatorname{row}}\left(  i+1,S\right)  -\mathrm{\operatorname{row}%
}\left(  i,S\right)  \right)  \geq2,
\]
and $S^{\left(  i\right)  }$, denoting the tableau obtained from $S$ by
exchanging $i$ and $i+1$, is an element of $\mathcal{Y}\left(  \tau\right)  $
and%
\begin{align*}
S\tau\left(  s_{i}\right)   &  =S^{\left(  i\right)  }+\frac{1}{c\left(
i,S\right)  -c\left(  i+1,S\right)  }S,\\
S\tau\left(  T_{i}\right)   &  =S^{\left(  i\right)  }+\dfrac{t-1}%
{1-t^{c\left(  i+1,S\right)  -c\left(  i,S\right)  }}S;
\end{align*}

4) $c\left(  i,S\right)  -c\left(  i+1,S\right)  \leq-2$, thus
$\mathrm{\operatorname{row}}\left(  i,S\right)  >\mathrm{\operatorname{row}%
}\left(  i+1,S\right)  $ and $\operatorname{col}\left(  i,S\right)
<\operatorname{col}\left(  i+1,S\right)  $ then with $b=c\left(  i,S\right)
-c\left(  i+1,S\right)  $,%
\begin{align*}
S\tau\left(  s_{i}\right)   &  =\left(  1-\frac{1}{b^{2}}\right)  S^{\left(
i\right)  }+\frac{1}{b}S,\\
S\tau\left(  T_{i}\right)   &  =\frac{t\left(  t^{b+1}-1\right)  \left(
t^{b-1}-1\right)  }{\left(  t^{b}-1\right)  ^{2}}S^{\left(  i\right)  }%
+\frac{t^{b}\left(  t-1\right)  }{t^{b}-1}S.
\end{align*}

The formulas in (4) are consequences of those in (3) by interchanging $S$ and
$S^{\left(  i\right)  }$ and applying the relations $\tau\left(  s_{i}\right)
^{2}=I$ and $\left(  \tau\left(  T_{i}\right)  +I\right)  \left(  \tau\left(
T_{i}\right)  -tI\right)  =0$ (where $I$ denotes the identity operator on
$V_{\tau}$).

There is a commutative set of Jucys-Murphy elements in both $\mathbb{Z}%
\emph{S}_{N}$ and $\mathcal{H}_{N}\left(  t\right)  $ and which are
diagonalized with respect to the basis $\mathcal{Y}\left(  \tau\right)  $
(with $1\leq i\leq N$ and $S\in\mathcal{Y}\left(  \tau\right)  $)%
\begin{align}
\omega_{i}  &  :=\sum_{j=i+1}^{N}\left(  i,j\right)  ,S\tau\left(  \omega
_{i}\right)  =c\left(  i,S\right)  S,\nonumber\\
\phi_{N}  &  =1,\phi_{i}=\frac{1}{t}T_{i}\phi_{i+1}T_{i},~S\tau\left(
\phi_{i}\right)  =t^{c\left(  i,S\right)  }S. \label{JMurph}%
\end{align}

The representation $\tau$ of $\mathcal{S}_{N}$ is unitary (orthogonal) when
$V_{\tau}$ is furnished with the inner product%
\[
\left\langle S,S^{\prime}\right\rangle _{0}:=\delta_{S,S^{\prime}{}^{\prime}%
}\times\prod_{\substack{1\leq i<j\leq N,\\c\left(  j,S\right)  -c\left(
i,S\right)  \geq2}}\left(  1-\frac{1}{\left(  c\left(  i,S\right)  -c\left(
j,S\right)  \right)  ^{2}}\right)  ,~S,S^{\prime}\in\mathcal{Y}\left(
\tau\right)  .
\]
The analogue for $\mathcal{H}_{N}\left(  t\right)  $ is
\[
\left\langle S,S^{\prime}\right\rangle _{0}:=\delta_{S,S^{\prime}{}^{\prime}%
}\times\prod_{\substack{1\leq i<j\leq N,\\c\left(  j,S\right)  -c\left(
i,S\right)  \geq2}}u\left(  t^{c\left(  i,S\right)  -c\left(  i,S\right)
}\right)  ,~S,S^{\prime}\in\mathcal{Y}\left(  \tau\right)  ,
\]
where%
\begin{equation}
u\left(  z\right)  :=\dfrac{\left(  t-z\right)  \left(  1-tz\right)  }{\left(
1-z\right)  ^{2}}. \label{uDef}%
\end{equation}
This form satisfies $\left\langle f\tau\left(  T_{i}\right)  ,g\right\rangle
_{0}=\left\langle f,g\tau\left(  T_{i}\right)  \right\rangle _{0}$ for $f,g\in
V_{\tau}$ and $1\leq i<N$.

\section{Representations and Operators on Polynomials\label{PolyRep}}

For $N\geq2,~x=\left(  x_{1},\ldots,x_{N}\right)  \in%
\mathbb{R}
^{N}$ . The cardinality of a set $E$ is denoted by $\#E$. For $\alpha
\in\mathbb{N}_{0}^{N}$ (a \textit{composition}) let $\left\vert \alpha
\right\vert :=\sum_{i=1}^{N}\alpha_{i}$, $x^{\alpha}:=\prod_{i=1}^{N}%
x_{i}^{\alpha_{i}}$, a monomial of degree $\left\vert \alpha\right\vert $. The
spaces of polynomials, respectively homogeneous, polynomials are
\begin{align*}
\mathcal{P} &  :=\mathrm{span}_{\mathbb{F}}\left\{  x^{\alpha}:\alpha
\in\mathbb{N}_{0}^{N}\right\}  ,\\
\mathcal{P}_{n} &  :=\mathrm{span}_{\mathbb{F}}\left\{  x^{\alpha}:\alpha
\in\mathbb{N}_{0}^{N},\left\vert \alpha\right\vert =n\right\}  ,~n\in%
\mathbb{N}
_{0}.
\end{align*}
For $\alpha\in\mathbb{N}_{0}^{N}$ let $\alpha^{+}$ denote the nonincreasing
rearrangement of $\alpha$. We use partial orders on $\mathbb{N}_{0}^{N}$ : for
$\alpha,\beta\in\mathbb{N}_{0}^{N}$, $\alpha\succ\beta$ ($\alpha$ dominates
$\beta$) means that $\alpha\neq\beta$ and $\sum_{i=1}^{j}\alpha_{i}\geq
\sum_{i=1}^{j}\beta_{i}$ for $1\leq j\leq N$; and $\alpha\vartriangleright
\beta$ means that $\left\vert \alpha\right\vert =\left\vert \beta\right\vert $
and either $\alpha^{+}\succ\beta^{+}$ or $\alpha^{+}=\beta^{+}$ and
$\alpha\succ\beta$. Also there is the rank function:%
\[
r_{\alpha}\left(  i\right)  :=\#\left\{  j:1\leq j\leq i,\alpha_{j}\geq
\alpha_{i}\right\}  +\#\left\{  j:i<j\leq N,\alpha_{j}>\alpha_{i}\right\}
,1\leq i\leq N;
\]
then $r_{\alpha}\in\mathcal{S}_{N}$ and $r_{\alpha}\left(  i\right)  =i$ for
all $i$ if and only if $\alpha=\alpha^{+}$.

The action of the symmetric group on polynomials is defined by%
\begin{align*}
xs_{i}  &  =\left(  x_{1}\ldots,\overset{i}{x}_{i+1},\overset{i+1}{x_{i}%
},\ldots,x_{N}\right) \\
p\left(  x\right)  s_{i}  &  =p\left(  xs_{i}\right)  ,1\leq i<N.
\end{align*}
For arbitrary transpositions $x\left(  i,j\right)  =\left(  \ldots,\overset
{i}{x}_{j},\ldots,\overset{j}{x}_{i},\ldots\right)  $ and $p\left(  x\right)
\left(  i,j\right)  =p\left(  x\left(  i,j\right)  \right)  .$There is a
subtlety (implicit inverse) involved due to acting on the right: for example
$p(x)s_{1}s_{2}=p\left(  xs_{1}\right)  s_{2}=p\left(  \left(  xs_{2}\right)
s_{1}\right)  $, that is, $p\left(  x_{1},x_{2},x_{3}\right)  s_{1}%
s_{2}=p\left(  x_{2},x_{1},x_{3}\right)  s_{2}=p\left(  x_{3},x_{1}%
,x_{2}\right)  $. In general $p\left(  x\right)  w=p\left(  xw^{-1}\right)  $
where $\left(  xw\right)  _{i}=x_{w^{-1}\left(  i\right)  }$ for all $i$.

The action of the Hecke algebra on polynomials is defined by%
\[
p\left(  x\right)  T_{i}=\left(  1-t\right)  x_{i+1}\frac{p\left(  x\right)
-p\left(  xs_{i}\right)  }{x_{i}-x_{i+1}}+tp\left(  xs_{i}\right)  .
\]
The defining relations can be verified straightforwardly. There are special
values: $x_{i}T_{i}=x_{i+1}$, $\left(  x_{i}+x_{i+1}\right)  T_{i}=t\left(
x_{i}+x_{i+1}\right)  $ and $\left(  tx_{i}-x_{i+1}\right)  T_{i}=-\left(
tx_{i}-x_{i+1}\right)  $. Also $pT_{i}=tp$ if and only if $ps_{i}=p$ , because
$tp-pT_{i}=\dfrac{tx_{i}-x_{i+1}}{x_{i}-x_{i+1}}\left(  p-ps_{i}\right)  $.

For a partition $\tau$ of $N$ let $\mathcal{P}_{\tau}:=\mathcal{P\otimes
}V_{\tau}$. The set $\left\{  x^{\alpha}\otimes S:\alpha\in\mathbb{N}_{0}%
^{N},S\in\mathcal{Y}\left(  \tau\right)  \right\}  $ is a basis of
$\mathcal{P}_{\tau}$. The representations of $\mathcal{S}_{N}$ and
$\mathcal{H}_{N}\left(  t\right)  $ on $\mathcal{P}_{\tau}$ are respectively
defined by the linear extension from the action on generators by%
\begin{align}
s_{i}  &  :p\left(  x\right)  \otimes S\rightarrow p\left(  xs_{i}\right)
\otimes S\tau\left(  s_{i}\right)  ,\nonumber\\
T_{i}  &  :p\left(  x\right)  \otimes S\rightarrow\left(  1-t\right)
x_{i+1}\frac{p\left(  x\right)  -p\left(  xs_{i}\right)  }{x_{i}-x_{i+1}%
}\otimes S+p\left(  xs_{i}\right)  \otimes S\tau\left(  T_{i}\right)  ,
\label{defTi}%
\end{align}
for $p\in\mathcal{P},S\in\mathcal{Y}\left(  \tau\right)  $ and $1\leq i<N$.
(For details and background for the vector-valued Macdonald polynomials see
\cite{DL2012}.)

\subsection{Jack polynomials\label{JPols}}

The Dunkl $\left\{  \mathcal{D}_{i}\right\}  $ and Cherednik-Dunkl $\left\{
\mathcal{U}_{i}\right\}  $ operators on $\mathcal{P}_{\tau}$ for
$p\in\mathcal{P},S\in\mathcal{Y}\left(  \tau\right)  $ and $1\leq i\leq N$,
are defined by%
\begin{align*}
\left(  p\left(  x\right)  \otimes S\right)  \mathcal{D}_{i} &  =\frac
{\partial}{\partial x_{i}}p\left(  x\right)  \otimes S+\kappa\sum_{j\neq
i,j=1}^{N}\frac{p\left(  x\right)  -p\left(  x\left(  i,j\right)  \right)
}{x_{i}-x_{j}}\otimes S\tau\left(  \left(  i,j\right)  \right)  ,\\
\left(  p\left(  x\right)  \otimes S\right)  \mathcal{U}_{i} &  =\left(
x_{i}p\left(  x\right)  \otimes S\right)  \mathcal{D}_{i}-\kappa\sum
_{j<i}p\left(  x\left(  i,j\right)  \right)  \otimes S\tau\left(  \left(
i,j\right)  \right)  .
\end{align*}
Each of the sets $\left\{  \mathcal{D}_{i}\right\}  $ and $\left\{
\mathcal{U}_{i}\right\}  $ consists of pairwise commuting elements. There is a
basis of $\mathcal{P}_{\tau}$ consisting of homogeneous polynomials each of
which is a simultaneous eigenfunction of $\left\{  \mathcal{U}_{i}\right\}  $;
these are the nonsymmetric Jack polynomials. For each $\left(  \alpha
,S\right)  \in\mathbb{N}_{0}^{N}\times\mathcal{Y}\left(  \tau\right)  $ there
is the polynomial%
\begin{equation}
J_{\alpha,S}=x^{\alpha}\otimes S\tau\left(  r_{\alpha}\right)  +\sum
_{\beta\vartriangleleft\alpha}x^{\beta}\otimes v_{\alpha,\beta,S}\left(
\kappa\right)  ,\label{Jxseries}%
\end{equation}
where $v_{\alpha,\beta,S}\left(  \kappa\right)  \in V_{\tau}$; these
coefficients are rational functions of $\kappa$. These polynomials satisfy%
\begin{align*}
J_{\alpha,S}\mathcal{U}_{i} &  =\zeta_{\alpha,S}\left(  i\right)  J_{\alpha
,S},\\
\zeta_{\alpha,S}\left(  i\right)   &  :=\alpha_{i}+1+\kappa c\left(
r_{\alpha}\left(  i\right)  ,S\right)  ,~1\leq i\leq N.
\end{align*}
The \textit{spectral vector} is $\left[  \zeta_{\alpha,S}\left(  i\right)
\right]  _{i=1}^{N}$. For detailed proofs see \cite{DL2011}.

We are concerned with the special case $\alpha=\left(  m,0,\ldots,0\right)
\in\mathbb{N}_{0}^{N}$. We apply formulas from \cite{Dunkl2018} to analyze
$J_{\alpha,S}\mathcal{D}_{i}$.

\begin{proposition}
\label{JDi0}(\cite[Cor. 6.2]{Dunkl2018}) Suppose $\left(  \beta,S\right)
\in\mathbb{N}_{0}^{N}\times\mathcal{Y}\left(  \tau\right)  $ and $\beta_{j}=0$
for $j\geq k$ with some fixed $k>1$ then $J_{\beta,S}\mathcal{D}_{j}=0$ for
all $j\geq k$.
\end{proposition}

The next result uses the inner product on Jack polynomials for partition
labels $\beta$. The Pochhammer symbol is $\left(  a\right)  _{n}=\prod
_{i=1}^{n}\left(  a+i-1\right)  $.

\begin{proposition}
Suppose $\beta\in\mathbb{N}_{0}^{N,+}$ and $S\in\mathcal{Y}\left(
\tau\right)  $ then%
\begin{align*}
\left\Vert J_{\beta,S}\right\Vert ^{2}  &  =\left\langle S,S\right\rangle
_{0}\prod_{i=1}^{N}\left(  1+\kappa c\left(  i,S\right)  \right)  _{\beta_{i}%
}\\
&  \times\prod_{1\leq i<j\leq N}\prod_{\ell=1}^{\beta_{i}-\beta_{j}}\left(
1-\left(  \frac{\kappa}{\ell+\kappa\left(  c\left(  i,S\right)  -c\left(
j,S\right)  \right)  }\right)  ^{2}\right)  .
\end{align*}

\end{proposition}

\begin{corollary}
Suppose $\alpha=\left(  m,0,\ldots,0\right)  $ then%
\[
\left\Vert J_{\alpha,S}\right\Vert ^{2}=\left\langle S,S\right\rangle
_{0}\left(  1+\kappa c\left(  1,S\right)  \right)  _{m}\prod_{j=2}^{N}%
\prod_{\ell=1}^{m}\left(  1-\left(  \frac{\kappa}{\ell+\kappa\left(  c\left(
1,S\right)  -c\left(  j,S\right)  \right)  }\right)  ^{2}\right)  .
\]

\end{corollary}

These norm formulas are results of Griffeth \cite{G2010} specialized to the
symmetric groups. The final ingredient for the formula is a special case of
\cite[Thm. 6.3]{Dunkl2018}.

\begin{proposition}
Suppose $\alpha=\left(  m,0,\ldots,0\right)  $ and $\widehat{\alpha}=\left(
m-1,0,\ldots,0\right)  $ then%
\begin{align}
J_{\alpha,S}\mathcal{D}_{1}  &  =\frac{\left\Vert J_{\alpha,S}\right\Vert
^{2}}{\left\Vert J_{\widehat{\alpha},S}\right\Vert ^{2}}J_{\widehat{\alpha}%
,S}\nonumber\\
&  =\left(  m+\kappa c\left(  1,S\right)  \right)  \prod_{j=2}^{N}\left(
1-\left(  \frac{\kappa}{m+\kappa\left(  c\left(  1,S\right)  -c\left(
j,S\right)  \right)  }\right)  ^{2}\right)  J_{\widehat{\alpha},S}.
\label{JDf1}%
\end{align}

\end{proposition}

\begin{proof}
The first line comes from \cite[Thm. 6.3]{Dunkl2018}. Then the norm ratios are
computed, which involves much cancellation.
\end{proof}

Denote the prefactor of $J_{\widehat{\alpha},S}$ in equation (\ref{JDf1}) by
$C_{S,m}\left(  \kappa\right)  $. Our interest is in the zeros of
$C_{S,m}\left(  \kappa\right)  $ as a function of $\kappa$. We will see that
$C_{S,m}\left(  \kappa\right)  $ depends only on $\tau$ and the location of
$1$ in $S$. The idea is to group entries of $S$ by row and use telescoping
properties. There is a simple formula (proven inductively)%
\[
\prod_{i=a}^{b}\frac{g\left(  i+1\right)  g\left(  i-1\right)  }{g\left(
i\right)  ^{2}}=\frac{g\left(  a-1\right)  g\left(  b+1\right)  }{g\left(
a\right)  g\left(  b\right)  }%
\]
where $g$ is a function on $\mathbb{Z}$ and $a\leq b$. For the present
application set $g\left(  i\right)  =m+\kappa\left(  c\left(  1,S\right)
-i\right)  $.

\begin{definition}
\label{tau_hat}The partition $\widehat{\tau}\in\mathbb{N}_{0}^{N,+}$ is
obtained from $\tau$ by removing the box $\left(  \operatorname{row}\left(
1,S\right)  ,\operatorname{col}\left(  1,S\right)  \right)  $: for $1\leq
i\leq\ell\left(  \tau\right)  $ set $\widehat{\tau}_{i}=\tau_{i}-1$ if
$\operatorname{row}\left(  1,S\right)  =i$ otherwise set $\widehat{\tau}%
_{i}=\tau_{i}$.
\end{definition}

The part of the product in $C_{S,m}\left(  \kappa\right)  $ coming from row
$\#i$ has $c\left(  j,S\right)  $ ranging from $1-i$ to $\widehat{\tau}_{i}-i$
so the corresponding subproduct is%
\[
\prod_{j=1-i}^{\widehat{\tau}_{i}-i}\frac{g\left(  j+1\right)  g\left(
j-1\right)  }{g\left(  j\right)  ^{2}}=\frac{g\left(  -i\right)  g\left(
\widehat{\tau}_{i}-i+1\right)  }{g\left(  1-i\right)  g\left(  \widehat{\tau
}_{i}-i\right)  }.
\]
Multiply these factors for $i=1,2,\ldots\ell\left(  \tau\right)  $; note that%
\[
\prod_{i=1}^{\ell\left(  \tau\right)  }\frac{g\left(  -i\right)  }{g\left(
1-i\right)  }=\frac{g\left(  -\ell\left(  \tau\right)  \right)  }{g\left(
0\right)  }=\frac{m+\kappa\left(  c\left(  1,S\right)  +\ell\left(
\tau\right)  \right)  }{m+\kappa c\left(  1,S\right)  },
\]
and thus%
\begin{equation}
C_{S,m}\left(  \kappa\right)  =\left(  m+\kappa\left(  c\left(  1,S\right)
+\ell\left(  \tau\right)  \right)  \right)  \prod_{i=1}^{\ell\left(
\tau\right)  }\frac{m+\kappa\left(  c\left(  1,S\right)  -\widehat{\tau}%
_{i}+i-1\right)  }{m+\kappa\left(  c\left(  1,S\right)  -\widehat{\tau}%
_{i}+i\right)  }. \label{CSk1}%
\end{equation}
As stated before the formula depends only on $\tau$ and the location of $1$ in
$S$. More simplification is possible due to telescoping if some $\widehat
{\tau}_{i}$'s are equal.

\begin{definition}
\label{T_tau}For $\widehat{\tau}$ as in Definition \ref{tau_hat} define the
increasing sequence $\mathcal{I}\left(  \widehat{\tau}\right)  =\left[
i_{1},i_{2},\ldots,i_{k}\right]  $ such that $i_{1}=1$, and $2\leq s\leq k$
implies $\widehat{\tau}_{i_{s}}<\widehat{\tau}_{i_{s-1}}$ and $\widehat{\tau
}_{j}=\widehat{\tau}_{i_{s-1}}$ for $i_{s-1}\leq j<i_{s}$. The last element
$i_{k}=\ell\left(  \tau\right)  +1$. Let $\mathcal{Z}\left(  \widehat{\tau
}\right)  =\left\{  \widehat{\tau}_{i_{s}}+1-i_{s}:1\leq s\leq k-1\right\}
\cup\left\{  -\ell\left(  \tau\right)  :\widehat{\tau}_{\ell\left(
\tau\right)  }\geq1\right\}  $ (the latter set is omitted when $\widehat{\tau
}_{\ell\left(  \tau\right)  }=0$).
\end{definition}

\begin{example}
Suppose $\widehat{\tau}=\left[  5,5,4,4,4,3,3,2,1\right]  $ then
$\mathcal{I}\left(  \widehat{\tau}\right)  =\left[  1,3,6,8,9,10\right]  $,
and $\mathcal{Z}\left(  \widehat{\tau}\right)  =\left\{
5,2,-2,-5,-7,-9\right\}  $. If $\widehat{\tau}=\left[
5,5,4,4,4,3,3,3,0\right]  $ then $\mathcal{I}\left(  \widehat{\tau}\right)
=\left[  1,3,6,9\right]  $ and $\mathcal{Z}\left(  \widehat{\tau}\right)
=\left\{  5,2,-2,-8\right\}  $.
\end{example}

Let $\widehat{S}$ denote the tableau formed by deleting the box $\left(
\operatorname{row}\left(  1,S\right)  ,\operatorname{col}\left(  1,S\right)
\right)  $ from $S$. The key property of $\mathcal{I}\left(  \widehat{\tau
}\right)  $ is that it controls the possible locations where a box containing
$1$ could be adjoined to $\widehat{S}$ to form a RSYT. These locations are
$\left\{  \left(  1,\widehat{\tau}_{1}+1\right)  ,\ldots,\left(
i_{s},\widehat{\tau}_{i_{s}}+1\right)  ,\ldots\right\}  $. If $\widehat{\tau
}_{\ell\left(  \tau\right)  }=0$ then the last location is $\left(
\ell\left(  \tau\right)  ,1\right)  $ otherwise it is $\left(  \ell\left(
\tau\right)  +1,1\right)  $. Thus $\mathcal{Z}\left(  \widehat{\tau}\right)  $
is the set of contents of locations in the list. Evaluate the part of the
product in formula (\ref{CSk1}) for the range $i_{s}\leq j<i_{s_{+1}}$ to
obtain%
\[
\prod_{j=i_{s}}^{i_{s+1}-1}\frac{m+\kappa\left(  c\left(  1,S\right)
-\widehat{\tau}_{i_{s}}+j-1\right)  }{m+\kappa\left(  c\left(  1,S\right)
-\widehat{\tau}_{i_{s}}+j\right)  }=\frac{m+\kappa\left(  c\left(  1,S\right)
-\left(  \widehat{\tau}_{i_{s}}+1-i_{s}\right)  \right)  }{m+\kappa\left(
c\left(  1,S\right)  -\left(  \widehat{\tau}_{i_{s}}+1-i_{s+1}\right)
\right)  }.
\]
This completes the proof of the following:

\begin{proposition}
\label{CSmprod}For $\widehat{\tau}$ and $\mathcal{I}\left(  \widehat{\tau
}\right)  $ as in Definitions \ref{tau_hat} and \ref{T_tau}%
\[
C_{S,m}\left(  \kappa\right)  =\left(  m+\kappa\left(  c\left(  1,S\right)
+\ell\left(  \tau\right)  \right)  \right)  \prod_{s=1}^{k-1}\frac
{m+\kappa\left(  c\left(  1,S\right)  -\left(  \widehat{\tau}_{i_{s}}%
+1-i_{s}\right)  \right)  }{m+\kappa\left(  c\left(  1,S\right)  -\left(
\widehat{\tau}_{i_{s}}+1-i_{s+1}\right)  \right)  },
\]
where $i_{k}=\ell\left(  \tau\right)  +1$.
\end{proposition}

If $\widehat{\tau}_{\ell\left(  \tau\right)  }=0$ then the entry at $\left(
\ell\left(  \tau\right)  ,1\right)  $ is $1$, $c\left(  1,S\right)
=1-\ell\left(  \tau\right)  $, $i_{k-1}=\ell\left(  \tau\right)  $ and the
last factor in the product (for $s=k-1$) equals $\frac{m}{m+\kappa}$, thus
cancelling out the leading factor $m+\kappa\left(  c\left(  1,S\right)
+\ell\left(  \tau\right)  \right)  =m+\kappa$.

\begin{lemma}
\label{not=}Suppose $1\leq a,b\leq k-1$ then $\widehat{\tau}_{i_{a}}-i_{a}%
\neq\widehat{\tau}_{i_{b}}-i_{b+1}$.
\end{lemma}

\begin{proof}
By construction the sequence $\left\{  \widehat{\tau}_{i_{a}}\right\}
_{a\geq1}$ is strictly decreasing and the sequence $\left\{  i_{a}\right\}
_{a\geq1}$ is strictly increasing. Suppose for some $a,b$ the equation
$\widehat{\tau}_{i_{a}}-i_{a}=\widehat{\tau}_{i_{b}}-i_{b+1}$ holds, that is,
$i_{a}-i_{b+1}=\widehat{\tau}_{i_{a}}-\widehat{\tau}_{i_{b}}$. Clearly $a=b$
or $a=b+1$ are impossible. Suppose $i_{a}-i_{b+1}>0$ then $b<b+1<a$ implying
$\widehat{\tau}_{i_{a}}-\widehat{\tau}_{i_{b}}<0$, a contradiction. Similarly
suppose $i_{a}-i_{b+1}<0$ then $a<b+1$, furthermore that $a<b$ since $a=b$ is
impossible, thus $\widehat{\tau}_{i_{a}}-\widehat{\tau}_{i_{b}}>0$, again a
contradiction. This completes the proof.
\end{proof}

\begin{proposition}
\label{Ckzeros}The set of zeros of $C_{m,S}\left(  \kappa\right)  $ is\newline
$\left\{  -\dfrac{m}{c\left(  1,S\right)  -z}:z\in\mathcal{Z}\left(
\widehat{\tau}\right)  ,z\neq c\left(  1,S\right)  \right\}  .$
\end{proposition}

\begin{proof}
None of the numerator factors in the product are cancelled out due to Lemma
\ref{not=}. The only possible cancellation occurs for $\widehat{\tau}%
_{\ell\left(  \tau\right)  }=0$ when $c\left(  1,S\right)  $ is the last entry
in the list $\mathcal{Z}\left(  \widehat{\tau}\right)  $.
\end{proof}

\begin{example}
\label{exJack}Let $N=7,\tau=\left[  5,5,5,4,4,2,2\right]  $ and $\left(
\operatorname{row}\left(  1,S\right)  ,\operatorname{col}\left(  1,S\right)
\right)  =\left(  3,5\right)  $, then $\widehat{\tau}=\left[
5,5,4,4,4,2,2\right]  $. The possible locations where the box containing $1$
could be adjoined to $\widehat{\tau}$ are $\left\{  \left(  1,6\right)
,\left(  3,5\right)  ,\left(  6,3\right)  ,\left(  8,1\right)  \right\}  $ so
that $\mathcal{Z}\left(  \widehat{\tau}\right)  =\left\{  5,2,-3,-7\right\}  $
and%
\[
C_{S,m}\left(  \kappa\right)  =\frac{m\left(  m-3\kappa\right)  \left(
m+5\kappa\right)  \left(  m+9\kappa\right)  }{\left(  m-\kappa\right)  \left(
m+3\kappa\right)  \left(  m+7\kappa\right)  }.
\]
Here is a sketch of $\widehat{\tau}$ marked by $\square$ and the possible
cells for the entry $1$%
\[%
\begin{array}
[c]{cccccc}%
\square & \square & \square & \square & \square & 1\\
\square & \square & \square & \square & \square & \\
\square & \square & \square & \square & 1 & \\
\square & \square & \square & \square &  & \\
\square & \square & \square & \square &  & \\
\square & \square & 1 &  &  & \\
\square & \square &  &  &  & \\
1 &  &  &  &  &
\end{array}
\]

\end{example}

In a later section we examine the relation to singular polynomials of the form
$J_{\alpha,S}$.

\subsection{Macdonald polynomials\label{Mpols}}

Adjoin the parameter $q$. To say that $\left(  q,t\right)  $ is generic means
that $q\neq1,q^{a}t^{b}\neq1$ for $a,b\in\mathbb{Z}$ and $-N\leq b\leq N$.
Besides the operators $T_{i}$ defined in (\ref{defTi}) we introduce (for
$p\in\mathcal{P},S\in\mathcal{Y}\left(  \tau\right)  $)
\begin{align*}
\omega &  :=T_{1}T_{2}\cdots T_{N-1}\text{,}\\
\left(  p\left(  x\right)  \otimes S\right)  w &  :=p\left(  qx_{N}%
,x_{1},\ldots,x_{N-1}\right)  \otimes S\tau\left(  \omega\right)  .
\end{align*}
The Cherednik $\left\{  \xi_{i}\right\}  $ and Dunkl $\left\{  \mathcal{D}%
_{i}\right\}  $ operators, for $1\leq i\leq N$, are defined by
\begin{align*}
\xi_{i} &  :=t^{i-N}T_{i-1}^{-1}\cdots T_{1}^{-1}wT_{N-1}\cdots T_{i},\\
\mathcal{D}_{N} &  :=\left(  1-\xi_{N}\right)  /x_{N},~\mathcal{D}_{i}%
:=\frac{1}{t}T_{i}\mathcal{D}_{i+1}T_{i}\text{.}%
\end{align*}
These definitions were given for the scalar case by Baker and Forrester
\cite{BF1997} and extended to vector-valued polynomials by Luque and the
author \cite{DL2012}. The operators $\left\{  \xi_{i}:1\leq i\leq N\right\}  $
commute pairwise, while the operators $\left\{  \mathcal{D}_{i}:1\leq i\leq
N\right\}  $ commute pairwise and map $\mathcal{P}_{n}\otimes V_{\tau}$ to
$\mathcal{P}_{n-1}\otimes V_{\tau}$ for $n\geq0$. A polynomial $p\in
\mathcal{P}_{\tau}$ is \textit{singular} for some particular value of $\left(
q,t\right)  $ if $p\mathcal{D}_{i}=0$ , evaluated at $\left(  q,t\right)  $,
for all $i$. There is a basis of $\mathcal{P}_{\tau}$ consisting of
homogeneous polynomials each of which is a simultaneous eigenfunction of
$\left\{  \mathcal{\xi}_{i}\right\}  $; these are the nonsymmetric Macdonald
polynomials. For each $\left(  \alpha,S\right)  \in\mathbb{N}_{0}^{N}%
\times\mathcal{Y}\left(  \tau\right)  $ there is the polynomial%
\[
M_{\alpha,S}=q^{q}t^{b}x^{\alpha}\otimes S\tau\left(  R_{\alpha}\right)
+\sum_{\beta\vartriangleleft\alpha}x^{\beta}\otimes v_{\alpha,\beta,S}\left(
q,t\right)  ,
\]
where $v_{\alpha,\beta,S}\left(  q,t\right)  \in V_{\tau}$ and $R_{\alpha
}:=\left(  T_{i_{i}}T_{i_{2}}\cdots T_{i_{m}}\right)  ^{-1}$ where
$\alpha.s_{i_{1}}s_{i_{2}}\cdots s_{i_{m}}=\alpha^{+}$ and there is no shorter
product $s_{j_{1}}s_{j_{2}}\cdots$having this property (that is $m=\#\left\{
\left(  i,j\right)  :i<j,\alpha_{i}<\alpha_{j}\right\}  $), $a,b\in\mathbb{Z}$
(see \cite[p. 19]{Dunkl2019} for the values of $a,b$, which are not needed
here), and
\begin{align*}
M_{\alpha,S}\xi_{i} &  =\widetilde{\zeta}_{\alpha,S}\left(  i\right)
M_{\alpha,S},\\
\widetilde{\zeta}_{\alpha,S}\left(  i\right)   &  =q^{\alpha_{i}}t^{c\left(
r_{\alpha}\left(  i\right)  ,S\right)  },~1\leq i\leq N.
\end{align*}
As before $\left[  \widetilde{\zeta}_{\alpha,S}\left(  i\right)  \right]
_{i=1}^{N}$ is called the \textit{spectral vector }(the tilde indicates
the\textit{ }$\left(  q,t\right)  $-version). We consider the special case
$\alpha=\left(  m,0,\ldots,0\right)  $.

\begin{proposition}
\label{MDi0}(\cite[Prop. 12]{Dunkl2019}) Suppose $\left(  \beta,S\right)
\in\mathbb{N}_{0}^{N}\times\mathcal{Y}\left(  \tau\right)  $ and $\beta_{j}=0$
for $j\geq k$ with some fixed $k>1$ then $M_{\beta,S}\mathcal{D}_{j}=0$ for
all $j\geq k$.
\end{proposition}

Adapting the proof of \cite[Lemma 5]{Dunkl2019} we show (recall (\ref{uDef})
$u\left(  z\right)  =\frac{\left(  t-z\right)  \left(  1-tz\right)  }{\left(
1-z\right)  ^{2}}$):

\begin{proposition}
Let $\alpha=\left(  m,0,\ldots\right)  $, $\alpha^{\prime}=\left(
0,0,\ldots,m\right)  $ and $S\in\mathcal{Y}\left(  \tau\right)  $ then%
\[
M_{\alpha,S}\mathcal{D}_{1}=t^{1-N}\prod\limits_{j=1}^{N-1}u\left(
q^{m}t^{c\left(  1,S\right)  -c\left(  j+1,S\right)  }\right)  M_{\alpha
^{\prime},S}\mathcal{D}_{N}T_{N-1}\cdots T_{1}\text{.}%
\]

\end{proposition}

The other ingredient is the affine step (from the Yang-Baxter graph, see
\cite{DL2012}, \cite[(3.14)]{Dunkl2019}): for $\beta\in\mathbb{N}_{0}^{N}$ set
$\beta\Phi:=\left(  \beta_{2},\beta_{3},\ldots,\beta_{N},\beta_{1}+1\right)  $
then $M_{\beta\Phi,S}=x_{N}\left(  M_{\beta,S}w\right)  .$ The spectral vector
of $\beta\Phi$ is $\left[  \widetilde{\zeta}_{\beta,S}\left(  2\right)
,\ldots,\widetilde{\zeta}_{\beta,S}\left(  N\right)  ,q\widetilde{\zeta
}_{\beta,S}\left(  1\right)  \right]  $. Observe $\widehat{\alpha}\Phi
=\alpha^{\prime}$ for $\widehat{\alpha}=\left(  m-1,0,\ldots\right)  $. By
definition%
\begin{align*}
M_{\alpha^{\prime},S}\mathcal{D}_{N}  &  =\frac{1}{x_{N}}\left\{
M_{\alpha^{\prime},S}\left(  1-\xi_{N}\right)  \right\}  =\frac{1}{x_{N}%
}\left(  1-\widetilde{\zeta}_{\alpha^{\prime},S}\left(  N\right)  \right)
M_{\alpha^{\prime},S}\\
&  =\left(  1-\widetilde{\zeta}_{\alpha^{\prime},S}\left(  N\right)  \right)
M_{\widehat{\alpha},S}w=\left(  1-q\widetilde{\zeta}_{\widehat{\alpha}%
,S}\left(  1\right)  \right)  M_{\widehat{\alpha},S}w.
\end{align*}
Furthermore $\left(  1-q\widetilde{\zeta}_{\widehat{\alpha},S}\left(
1\right)  \right)  =\left(  1-q^{m}t^{c\left(  1,S\right)  }\right)  $ and
$wT_{N-1}\cdots T_{1}=t^{N-1}\xi_{1}$ so that $M_{\widehat{\alpha},S}\xi
_{1}=q^{m-1}t^{c\left(  1,S\right)  }M_{\widehat{\alpha},S}\ $.

\begin{proposition}
Let $\alpha=\left(  m,0,\ldots\right)  $, $\widehat{\alpha}=\left(
m-1,0,\ldots,0\right)  $ and $S\in\mathcal{Y}\left(  \tau\right)  $ then%
\begin{equation}
M_{\alpha,S}\mathcal{D}_{1}=q^{m-1}t^{c\left(  1,S\right)  }\left(
1-q^{m}t^{c\left(  1,S\right)  }\right)  \prod\limits_{j=2}^{N}u\left(
q^{m}t^{c\left(  1,S\right)  -c\left(  j,S\right)  }\right)  M_{\widehat
{\alpha},S}. \label{MDf1}%
\end{equation}

\end{proposition}

This is very similar to the Jack case (\ref{JDf1}) and the same telescoping
argument will be used. Denote the factor of $M_{\widehat{\alpha},S}$ in
(\ref{MDf1}) by $C_{S,m}\left(  q,t\right)  $. Set $g\left(  i\right)
=1-q^{m}t^{c\left(  1,S\right)  -i}$ for $i\in\mathbb{Z}$ then%
\[
u\left(  q^{m}t^{c\left(  1,S\right)  -c\left(  j,S\right)  }\right)
=t\frac{g\left(  c\left(  j,S\right)  -1\right)  g\left(  c\left(  j,S\right)
+1\right)  }{g\left(  c\left(  j,S\right)  \right)  ^{2}}.
\]
With the same notation for $\widehat{\tau}$ as in Definition \ref{tau_hat}%
\begin{align*}
C_{S,m}\left(  q,t\right)   &  =q^{m-1}t^{c\left(  1,S\right)  +N-1}\left(
1-q^{m}t^{c\left(  1,S\right)  }\right)  \prod_{i=1}^{\ell\left(  \tau\right)
}\frac{g\left(  -i\right)  }{g\left(  1-i\right)  }\frac{g\left(
\widehat{\tau}_{i}-i+1\right)  }{g\left(  \widehat{\tau}_{i}-i\right)  }\\
&  =q^{m-1}t^{c\left(  1,S\right)  +N-1}\left(  1-q^{m}t^{c\left(  1,S\right)
}\right)  \frac{g\left(  -\ell\left(  \tau\right)  \right)  }{g\left(
0\right)  }\prod_{i=1}^{\ell\left(  \tau\right)  }\frac{g\left(  \widehat
{\tau}_{i}-i+1\right)  }{g\left(  \widehat{\tau}_{i}-i\right)  }\\
&  =q^{m-1}t^{c\left(  1,S\right)  +N-1}\left(  1-q^{m}t^{c\left(  1,S\right)
+\ell\left(  \tau\right)  }\right)  \prod_{i=1}^{\ell\left(  \tau\right)
}\frac{g\left(  \widehat{\tau}_{i}-i+1\right)  }{g\left(  \widehat{\tau}%
_{i}-i\right)  }.
\end{align*}
The same computational scheme as in Proposition \ref{CSmprod} proves the following:

\begin{proposition}
For $\widehat{\tau}$ and $\mathcal{I}\left(  \widehat{\tau}\right)  $ as in
Definitions \ref{tau_hat} and \ref{T_tau}%
\[
C_{S,m}\left(  q,t\right)  =q^{m-1}t^{c\left(  1,S\right)  +N-1}\left(
1-q^{m}t^{c\left(  1,S\right)  +\ell\left(  \tau\right)  }\right)  \prod
_{s=1}^{k-1}\frac{1-q^{m}t^{c\left(  1,S\right)  -\left(  \widehat{\tau
}_{i_{s}}+1-i_{s}\right)  }}{1-q^{m}t^{c\left(  1,S\right)  -\left(
\widehat{\tau}_{i_{s}}+1-i_{s+1}\right)  }},
\]
where $i_{k}=\ell\left(  \tau\right)  +1$.
\end{proposition}

If $\widehat{\tau}_{\ell\left(  \tau\right)  }=0$ then the entry at $\left(
\ell\left(  \tau\right)  ,1\right)  $ is $1$, $c\left(  1,S\right)
=1-\ell\left(  \tau\right)  $, $i_{k-1}=\ell\left(  \tau\right)  $ and the
last factor in the product (for $s=k-1$) equals $\dfrac{1-q^{m}}{1-q^{m}t}$,
thus cancelling out the leading factor $1-q^{m}t^{c\left(  1,S\right)
+\ell\left(  \tau\right)  }=1-q^{m}t$.

\begin{proposition}
\label{Cqtzeros}The set of zeros of $C_{m,S}\left(  q,t\right)  $
is\newline\ $\left\{  q^{m}t^{c\left(  1,S\right)  -z}=1:z\in\mathcal{Z}%
\left(  \widehat{\tau}\right)  ,z\neq c\left(  1,S\right)  \right\}  $.
\end{proposition}

\begin{proof}
None of the numerator factors in the product are cancelled out due to Lemma
\ref{not=}. The only possible cancellation occurs for $\widehat{\tau}%
_{\ell\left(  \tau\right)  }=0$ when $c\left(  1,S\right)  $ is the last entry
in the list $\mathcal{Z}\left(  \widehat{\tau}\right)  $.
\end{proof}

\begin{example}
Let $N=7,\tau=\left[  5,5,4,4,4,3,2\right]  $ and $\left(  \operatorname{row}%
\left(  1,S\right)  ,\operatorname{col}\left(  1,S\right)  \right)  =\left(
6,3\right)  $, then $\widehat{\tau}=\left[  5,5,4,4,4,2,2\right]  $. This is
the same $\widehat{\tau}$ as in Example \ref{exJack}, and $\mathcal{Z}\left(
\widehat{\tau}\right)  =\left\{  5,2,-3,-7\right\}  $. The same diagram
applies here. Then%
\[
C_{S,m}\left(  q,t\right)  =q^{m-1}t^{4}\left(  1-q^{m}\right)  \frac{\left(
1-q^{m}t^{-8}\right)  \left(  1-q^{m}t^{-5}\right)  \left(  1-q^{m}%
t^{4}\right)  }{\left(  1-q^{m}t^{-6}\right)  \left(  1-q^{m}t^{-2}\right)
\left(  1-q^{m}t^{2}\right)  }.
\]

\end{example}

In the next section we will see under what conditions $M_{\alpha,S}$ is singular.

\section{Singular Polynomials\label{SingPol}}

For $\alpha=\left(  m,0,\ldots\right)  \in\mathbb{N}_{0}^{N,+},\widehat
{\alpha}=\left(  m-1,0,\ldots\right)  $ and $S\in\mathcal{Y}\left(
\tau\right)  $ we have shown%
\begin{align*}
J_{\alpha,S}\mathcal{D}_{i} &  =0,M_{\alpha,S}\mathcal{D}_{i}=0,~2\leq i\leq
N;\\
J_{\alpha,S}\mathcal{D}_{1} &  =C_{S,m}\left(  \kappa\right)  J_{\widehat
{\alpha},S},\\
M_{\alpha,S}\mathcal{D}_{1} &  =C_{S,m}\left(  q,t\right)  M_{\widehat{\alpha
},S},
\end{align*}
and we determined the zeros of $C_{S,m}\left(  \kappa\right)  $ and
$C_{S,m}\left(  q,t\right)  .$But not all zeros lead to singular polynomials
because, in general the coefficients of $J_{\beta,S}$ (with respect to the
monomial basis $\left\{  x^{\gamma}\otimes S^{\prime}\right\}  $) have
denominators of the form $a+b\kappa$ and the coefficients of $M_{\beta,S}$
have denominators of the form $1-q^{a}t^{b}$ where $a,b\in\mathbb{Z}$ and
$\left\vert b\right\vert \leq N$. Thus to be able to substitute $\kappa
=\kappa_{0}$, a zero of $C_{S,m}\left(  \kappa\right)  $, or $\left(
q,t\right)  =\left(  q_{0},t_{0}\right)  $, a zero of $C_{S,m}\left(
q,t\right)  $, in equations (\ref{JDf1}) and (\ref{MDf1}) to conclude that
$J_{\alpha,S}$ or $M_{\alpha,S}$ are singular it is necessary to show that
neither $J_{\alpha,S}$ or $J_{\widehat{\alpha},S}$ have a pole at
$\kappa=\kappa_{0}$; the analogous requirement applies to $M_{\alpha,S}$ and
$M_{\widehat{\alpha},S}$. From the triangularity of $J_{\beta,S}$ and
$M_{\beta,S}$ with respect to the monomial basis we can deduce that%
\begin{align*}
x^{\lambda}\otimes S &  =J_{\lambda,S}+\sum_{\gamma\vartriangleleft
\lambda,S^{\prime}\in\mathcal{Y}\left(  \tau\right)  }b_{\lambda
,\gamma,S,S^{\prime}}\left(  \kappa\right)  J_{\gamma,S^{\prime}},\\
x^{\lambda}\otimes S &  =cM_{\lambda,S}+\sum_{\gamma\vartriangleleft
\lambda,S^{\prime}\in\mathcal{Y}\left(  \tau\right)  }b_{\lambda
,\gamma,S,S^{\prime}}\left(  q,t\right)  M_{\gamma,S^{\prime}},
\end{align*}
where $\lambda\in\mathbb{N}_{0}^{N,+}$, the coefficients $b_{\lambda
,\gamma,S,S^{\prime}}\left(  \kappa\right)  ,b_{\lambda,\gamma,S,S^{\prime}%
}\left(  q,t\right)  $ are rational functions of $\kappa,\left(  q,t\right)  $
respectively and $c=q^{a}t^{b}$ for some integers $a,b$. If one can show that
for each $\left(  \gamma,S^{\prime}\right)  $ with $\gamma\vartriangleleft
\lambda$ that the spectral vector is distinct from that of $\left(
\lambda,S\right)  $, that is, $\left[  \zeta_{\gamma,S^{\prime}}\left(
i\right)  \right]  _{i=1}^{N}\neq\left[  \zeta_{\lambda,S}\left(  i\right)
\right]  _{i=1}^{N}$ when evaluated at the specific values of $\kappa$ or
$\left(  q,t\right)  $ (with $\widetilde{\zeta}$) then $J_{\lambda,S}$,
respectively $M_{\lambda,S}$, do not have a pole there. The following is a
device for analyzing possibly coincident spectral vectors.

\begin{definition}
Let $\left(  \beta,S\right)  ,\left(  \gamma,S^{\prime}\right)  \in
\mathbb{N}_{0}^{N}\times\mathcal{Y}\left(  \tau\right)  $ such that
$\beta\vartriangleright\gamma$, and let $m,n\in\mathbb{Z}$ with $m\geq
1,n\neq0$. Then $\left[  \left(  \beta,S\right)  ,\left(  \gamma,S^{\prime
}\right)  \right]  $ is an $\left(  m,n\right)  $-critical pair if there is
$v\in\mathbb{Z}^{N}$ such that $\beta_{i}-\gamma_{i}=mv_{i}$ and $c\left(
r_{\beta}\left(  i\right)  ,S\right)  -c\left(  r_{\gamma}\left(  i\right)
,S^{\prime}\right)  =nv_{i}$ for $1\leq i\leq N$.
\end{definition}

\begin{lemma}
\label{Cpair}Let $\left(  \beta,S\right)  ,\left(  \gamma,S^{\prime}\right)
\in\mathbb{N}_{0}^{N}\times\mathcal{Y}\left(  \tau\right)  $ such that
$\beta\vartriangleright\gamma$ and $\zeta_{\beta,S}\left(  i\right)
=\zeta_{\gamma,S^{\prime}}\left(  i\right)  $ for all $i$ when $\kappa
=-\frac{m}{n}$, with $\gcd\left(  m,n\right)  =1$, then $\left[  \left(
\beta,S\right)  ,\left(  \gamma,S^{\prime}\right)  \right]  $ is an $\left(
m,n\right)  $-critical pair.
\end{lemma}

\begin{proof}
By hypothesis $\left(  1+\beta_{i}-\frac{m}{n}c\left(  r_{\beta}\left(
i,S\right)  \right)  \right)  =\left(  1+\gamma_{i}-\frac{m}{n}c\left(
r_{\gamma}\left(  i,S^{\prime}\right)  \right)  \right)  $ for $1\leq i\leq
N$; thus%
\begin{align*}
\beta_{i}-\gamma_{i}  &  =\frac{m}{n}\left(  r_{\beta}\left(  i,S\right)
-c\left(  r_{\gamma}\left(  i,S^{\prime}\right)  \right)  \right)  ,\\
n\left(  \beta_{i}-\gamma_{i}\right)   &  =m\left(  r_{\beta}\left(
i,S\right)  -c\left(  r_{\gamma}\left(  i,S^{\prime}\right)  \right)  \right)
.
\end{align*}
From $\gcd\left(  m,n\right)  =1$ it follows that $\beta_{i}-\gamma_{i}%
=mv_{i}$ for some $v_{i}\in\mathbb{Z}$ and thus $r_{\beta}\left(  i,S\right)
-c\left(  r_{\gamma}\left(  i,S^{\prime}\right)  \right)  =nv_{i}$.
\end{proof}

Now we specialize to $\alpha=\left(  m,0,\ldots\right)  $ as in Subsection
\ref{JPols} and $n$ satisfying $C_{S,m}\left(  -\frac{m}{n}\right)  =0$. By
Proposition \ref{Ckzeros} this is equivalent to $n=c\left(  1,S\right)  -z$
with $z\in\mathcal{Z}\left(  \widehat{\tau}\right)  $.

\begin{proposition}
\label{noCpair}There are no $\left(  m,n\right)  $-critical pairs $\left[
\left(  \alpha,S\right)  ,\left(  \gamma,S^{\prime}\right)  \right]  $.
\end{proposition}

\begin{proof}
Suppose that $\gamma\trianglelefteq\alpha$ and $\alpha_{i}-\gamma_{i}=mv_{i}%
$,$~c\left(  i,S\right)  -c\left(  r_{\gamma}\left(  i\right)  ,S^{\prime
}\right)  =nv_{i}$ with $v_{i}\in\mathbb{Z}$, and $1\leq i\leq N$. From
$\left\vert \gamma\right\vert =\left\vert \alpha\right\vert =m$ and
$\alpha_{j}=m$ or $=0$ it follows that $\gamma_{k}=m$ for some $k$ and
$\gamma_{i}=0$ for $i\neq k$. If $k=1$ then $z_{i}=0$ for all $i$ and
$c\left(  i,S\right)  =c\left(  r_{\gamma}\left(  i\right)  ,S^{\prime
}\right)  =c\left(  i,S^{\prime}\right)  $, because $\gamma\in\mathbb{N}%
_{0}^{N,+}$. The content vector determines $S^{\prime}$ uniquely and thus
$S^{\prime}=S$ and $\gamma=\alpha$. Now suppose $k>1$ then $v_{1}=1,v_{k}=-1$
and $v_{i}=0$ otherwise. The respective content vectors are%
\begin{align*}
\left[  c\left(  i,S\right)  \right]  _{i=1}^{N} &  =\left[  c\left(
1,S\right)  ,c\left(  2,S\right)  ,\ldots,c\left(  k,S\right)  ,c\left(
k+1,S\right)  ,\ldots,c\left(  N,S\right)  \right]  ,\\
\left[  c\left(  r_{\gamma}\left(  i\right)  ,S^{\prime}\right)  \right]
_{i=1}^{N} &  =\left[  c\left(  2,S^{\prime}\right)  ,c\left(  3,S^{\prime
}\right)  ,\ldots,c\left(  1,S^{\prime}\right)  ,c\left(  k+1,S^{\prime
}\right)  ,\ldots,c\left(  N,S^{\prime}\right)  \right]  .
\end{align*}
The hypothesis on $\gamma$ implies $c\left(  i,S^{\prime}\right)  =c\left(
i-1,S\right)  $ for $i=3\leq i\leq k$, $c\left(  i,S^{\prime}\right)
=c\left(  i,S\right)  $ for $k+1\leq i\leq N$, and $c\left(  2,S^{\prime
}\right)  =c\left(  1,S\right)  -n$, $c\left(  1,S^{\prime}\right)  =c\left(
k,S\right)  +n$. Since $S$ and $S^{\prime}$ are both of shape $\tau$ the two
content vectors are permutations of each other. The list of values $\left[
c\left(  3,S^{\prime}\right)  ,\ldots,c\left(  N,S^{\prime}\right)  \right]  $
agrees with $\left[  c\left(  2,S\right)  ,\ldots,c\left(  k-1,S\right)
,c\left(  k+1,S\right)  \ldots,c\left(  N,S\right)  \right]  $ thus $\left[
c\left(  1,S\right)  ,c\left(  k,S\right)  \right]  $ and $\left[  c\left(
1,S^{\prime}\right)  ,c\left(  2,S^{\prime}\right)  \right]  $ contain the
same two numbers. Since $c\left(  2,S^{\prime}\right)  =c\left(  1,S\right)
-n\neq c\left(  1,S\right)  $ the equation $c\left(  1,S\right)  =c\left(
1,S^{\prime}\right)  $ must hold. The possible locations of the entry $1$ in a
RSYT must have different contents (else they would be on the same diagonal
$\left\{  \left(  i,j\right)  :j-i=c\left(  1,S\right)  \right\}  $). Thus
$\left(  \operatorname{row}\left(  1,S^{\prime}\right)  ,\operatorname{col}%
\left(  1,S^{\prime}\right)  \right)  =\left(  \operatorname{row}%
(1,S,\operatorname{col}\left(  1,S\right)  \right)  $ and $S$ and $S^{\prime}$
lead to the same $\widehat{\tau}$ (the partition formed by removing the cell
of $1$ from $\tau).$By construction $n=z$ for some $z\in\mathcal{Z}\left(
\widehat{\tau}\right)  $, and $z$ determines a cell $\left(  i_{s}%
,\widehat{\tau}_{i_{s}}+1\right)  $ where $1$ can be attached to the part of
$S^{\prime}$ containing $\left\{  2,3,\ldots,N\right\}  $ to form a new RSYT
$S^{\prime\prime}$. By construction $c\left(  1,S^{\prime\prime}\right)
=z=c\left(  1,S\right)  -n=c\left(  2,S^{\prime}\right)  =c\left(
2,S^{\prime\prime}\right)  .$It is impossible for $c\left(  1,S^{\prime\prime
}\right)  =c\left(  2,S^{\prime\prime}\right)  $ for any RSYT, thus
$\gamma\neq\alpha$ can not occur.
\end{proof}

The same problem for $\widehat{\alpha}=\left(  m-1,0,\ldots\right)  $ is
almost trivial.

\begin{lemma}
\label{JMpoles2}Suppose $S^{\prime}\in\mathcal{Y}\left(  \tau\right)  $,
$\left\vert \gamma\right\vert =m-1$ and $\widehat{\alpha}_{i}-\gamma
_{i}=mv_{i}$,$~c\left(  i,S\right)  -c\left(  r_{\gamma}\left(  i\right)
,S^{\prime}\right)  =nv_{i}$ with $v_{i}\in\mathbb{Z}$, and $1\leq i\leq N$.
Then $\left(  \widehat{\alpha},S\right)  =\left(  \gamma,S^{\prime}\right)  $.
\end{lemma}

\begin{proof}
The hypothesis $\left\vert \gamma\right\vert =m-1$ implies $\gamma_{i}\leq
m-1$ and thus $\left\vert \widehat{\alpha}_{i}-\gamma_{i}\right\vert \leq m-1$
for all $i$. This implies $v_{i}=0$ for all $i$ implying $\gamma
=\widehat{\alpha}$ and $c\left(  j,S\right)  =c\left(  j,S^{\prime}\right)  $
for all $j$, thus $S=S^{\prime}$.
\end{proof}

\begin{proposition}
\label{Jpoles1}Suppose $\left(  \beta,S\right)  \in\mathbb{N}_{0}^{N}%
\times\mathcal{Y}\left(  \tau\right)  $, $\gcd\left(  m,n\right)  =1$ and
there are no $\left(  m,n\right)  $-critical pairs $\left[  \left(
\beta,S\right)  ,\left(  \gamma,S^{\prime}\right)  \right]  $ then
$J_{\beta,S}$ has no poles at $\kappa=-\frac{m}{n}$.
\end{proposition}

\begin{proof}
By the triangularity of formula (\ref{Jxseries}) there is an expansion%
\[
x^{\beta}\otimes S\tau\left(  r_{\beta}\right)  =J_{\beta,S}+\sum
_{\gamma\vartriangleleft\beta,S^{\prime}\in\mathcal{Y}\left(  \tau\right)
}b_{\beta,\gamma,S,S^{\prime}}\left(  \kappa\right)  J_{\gamma,S^{\prime}}.
\]
By Lemma \ref{Cpair} for each $\gamma\vartriangleleft\beta,S^{\prime}%
\in\mathcal{Y}\left(  \tau\right)  $ there is at least one $i=i\left[
\gamma,S^{\prime}\right]  $ such that $\zeta_{\beta,S}\left(  i\right)
-\zeta_{\gamma,S^{\prime}}\left(  i\right)  \neq0$ when $\kappa=-\frac{m}{n}$.
Define an operator%
\[
\mathcal{T}:=\prod_{\gamma\vartriangleleft\beta,S^{\prime}\in\mathcal{Y}%
\left(  \tau\right)  }\frac{\mathcal{U}_{i\left[  \gamma,S^{\prime}\right]
}-\zeta_{\gamma,S^{\prime}}\left(  i\left[  \gamma,S^{\prime}\right]  \right)
}{\zeta_{\beta,S}\left(  i\left[  \gamma,S^{\prime}\right]  \right)
-\zeta_{\gamma,S^{\prime}}\left(  i\left[  \gamma,S^{\prime}\right]  \right)
}.
\]
Then $J_{\beta,S}\mathcal{T}=J_{\beta,S}$ and each $J_{\gamma,S^{\prime}}$
(with $\gamma\vartriangleleft\beta$) is annihilated by at least one factor of
$\mathcal{T}$. Thus $J_{\beta,S}=\left(  x^{\beta}\otimes S\tau\left(
r_{\beta}\right)  \right)  \mathcal{T}$, a polynomial whose coefficients have
denominators which are factors of\linebreak\ $\prod\limits_{\gamma
\vartriangleleft\beta,S^{\prime}\in\mathcal{Y}\left(  \tau\right)  }\left(
\zeta_{\beta,S}\left(  i\left[  \gamma,S^{\prime}\right]  \right)
-\zeta_{\gamma,S^{\prime}}\left(  i\left[  \gamma,S^{\prime}\right]  \right)
\right)  $. By construction of $\left\{  i\left[  \gamma,S^{\prime}\right]
\right\}  $ this product does not vanish at $\kappa=-\frac{m}{n}$.
\end{proof}

We are ready for the main result on Jack polynomials.

\begin{theorem}
Suppose $\alpha=\left(  m,0,\ldots\right)  ,S\in\mathcal{Y}\left(
\tau\right)  $ and $\mathcal{Z}\left(  \widehat{\tau}\right)  $ is as in
Definition \ref{T_tau}. Further suppose $z\in\mathcal{Z}\left(  \widehat{\tau
}\right)  $, $n:=c\left(  1,S\right)  -z\neq0$ and $\gcd\left(  m,n\right)
=1$ then $J_{\alpha,S}$ is a singular polynomial for $\kappa=-\frac{m}{n}$.
\end{theorem}

\begin{proof}
From Proposition \ref{JDi0} $J_{\alpha,S}\mathcal{D}_{j}=0$ for $2\leq j\leq
N$ and $J_{\alpha,S}\mathcal{D}_{1}=C_{S,m}\left(  \kappa\right)
J_{\widehat{\alpha},S}$, where $\widehat{\alpha}=\left(  m-1,0,\ldots\right)
.$By Propositions \ref{noCpair}, \ref{Jpoles1} and Lemma \ref{JMpoles2}
$J_{\alpha,S}$ and $J_{\widehat{\alpha},S}$ do not have poles at
$\kappa=-\frac{m}{n}$. Furthermore $C_{S,m}\left(  -\frac{m}{n}\right)  =0$
and thus $J_{\alpha,S}\mathcal{D}_{1}=0$ at $\kappa=-\frac{m}{n}$.
\end{proof}

To set up the analogous results for Macdonald polynomials consider the
differences between two spectral vectors: $\widetilde{\zeta}_{\beta,S}\left(
i\right)  -\widetilde{\zeta}_{\gamma,S^{\prime}}\left(  i\right)
=q^{\beta_{i}}t^{c\left(  r_{\beta}\left(  i\right)  ,S\right)  }%
-q^{\gamma_{i}}t^{c\left(  r_{\gamma}\left(  i\right)  ,S^{\prime}\right)
}=q^{\gamma_{i}}t^{c\left(  r_{\gamma}\left(  i\right)  ,S^{\prime}\right)
}\left(  q^{\beta_{i}-\gamma_{i}}t^{c\left(  r_{\gamma}\left(  i\right)
,S^{\prime}\right)  -c\left(  r_{\gamma}\left(  i\right)  ,S^{\prime}\right)
}-1\right)  $. To relate this to $\left(  m,n\right)  $-critical pairs we
specify a condition on $\left(  q,t\right)  $ which implies $a=mv$ and $b=nv$
for some $v\in\mathbb{Z}$ when $q^{a}t^{b}=1$.

\begin{definition}
Suppose $m,n$ are integers such $m\geq1,n\neq0$ and $\gcd\left(  m,n\right)
=g\geq1.$ Let $u\in\mathbb{C}\backslash\left\{  0\right\}  $ such that $u$ is
not a root of unity and $\omega=\exp\left(  \frac{2\pi\mathrm{i}k}{m}\right)
$ with $\gcd\left(  k,g\right)  =1$. Define $\varpi=\left(  q,t\right)
=\left(  \omega u^{-n/g},u^{m/g}\right)  $.
\end{definition}

\begin{lemma}
\label{qtab}Suppose $a,b$ are integers such that $q^{a}t^{b}=1$ at $\left(
q,t\right)  =\varpi$ then $a=mv,b=nv$ for some $v\in\mathbb{Z}$.
\end{lemma}

\begin{proof}
By hypothesis%
\[
1=\left(  \omega u^{-n/g}\right)  ^{a}\left(  u^{m/g}\right)  ^{b}=\omega
^{a}u^{\left(  -na+mb\right)  /g}.
\]
Since $u$ is not a root of unity it follows that $-a\left(  \frac{n}%
{g}\right)  +b\left(  \frac{m}{g}\right)  =0$ but $\gcd\left(  \frac{n}%
{g},\frac{m}{g}\right)  =1$ and thus $\frac{m}{g}$ divides $a.$ Write
$a=\left(  \frac{m}{g}\right)  c$ for some integer $c$ then $1=\omega^{a}%
=\exp\left(  \frac{2\pi\mathrm{i}k}{m}\frac{mc}{g}\right)  =\exp\left(
\frac{2\pi\mathrm{i}kc}{g}\right)  .$This implies $c=vg$ with $v\in\mathbb{Z}$
because $\exp\left(  \frac{2\pi\mathrm{i}k}{g}\right)  $ is a primitive
$g^{th}$ root of unity. Thus $a=\left(  \frac{m}{g}\right)  vg=mv$ and
$b=\frac{n}{m}a=nv$.
\end{proof}

\begin{remark}
All the possible values of $\varpi$ are included when (1) $g>1$ and
$\omega=\exp\left(  \frac{2\pi\mathrm{i}k}{m}\right)  $ with $\gcd\left(
k,g\right)  =1$ and $1\leq k<g$ (2) $g=1$ and $\omega=1$. To prove this let
$u=\phi v$ with $\phi=\exp\left(  \frac{2\pi\mathrm{i}gl}{m}\right)  $ and
$l\in\mathbb{Z}$ so that $u^{m/g}=v^{m/g}$. Then $q=\omega\phi^{-n/g}%
v^{-n/g}=\exp\left(  2\pi\mathrm{i}\left(  \frac{k-nl}{m}\right)  \right)
v^{-n/g}$. Since $\gcd\left(  m,n\right)  =g$ there are integers $s,s^{\prime
}$ such that $s^{\prime}m+sn=g$. Set $l=s"s$ (with $s^{\prime\prime}%
\in\mathbb{Z})$ then $k-nl=k-s^{\prime\prime}g-s^{\prime\prime}s^{\prime}m$;
thus $\omega\phi^{-n/g}=\exp\left(  2\pi\mathrm{i}\left(  \frac{k-s^{\prime
\prime}g}{m}\right)  \right)  $. If $g>1$ then let $s^{\prime\prime
}=\left\lfloor \frac{k}{g}\right\rfloor +1$ implying $1\leq k-s^{\prime\prime
}g<g$, while if $g=1$ set $s^{\prime\prime}=k$.
\end{remark}

\begin{example}
Suppose $m=8$ and $n=-12$, then $g=4$ and the possible values of $\varpi$ are
$\left(  \exp\left(  \frac{\pi\mathrm{i}}{4}\right)  u^{3},u^{2}\right)  $ and
$\left(  \exp\left(  \frac{3\pi\mathrm{i}}{4}\right)  u^{3},u^{2}\right)  $
where $u$ is not a root of unity.
\end{example}

We will use this result to produce singular polynomials $M_{\alpha,S}$ for
$\left(  q,t\right)  =\varpi$.

\begin{lemma}
Let $\left(  \beta,S\right)  ,\left(  \gamma,S^{\prime}\right)  \in
\mathbb{N}_{0}^{N}\times\mathcal{Y}\left(  \tau\right)  $ such that
$\beta\vartriangleright\gamma$ and $\widetilde{\zeta}_{\beta,S}\left(
i\right)  =\widetilde{\zeta}_{\gamma,S^{\prime}}\left(  i\right)  $ for all
$i$ when $\left(  q,t\right)  =\varpi$ then $\left[  \left(  \beta,S\right)
,\left(  \gamma,S^{\prime}\right)  \right]  $ is an $\left(  m,n\right)
$-critical pair.
\end{lemma}

\begin{proof}
The equation $\widetilde{\zeta}_{\beta,S}\left(  i\right)  =\widetilde{\zeta
}_{\gamma,S^{\prime}}\left(  i\right)  $ is $q^{\beta_{i}}t^{c\left(
r_{\beta}\left(  i\right)  ,S\right)  }=q^{\gamma_{i}}t^{c\left(  r_{\gamma
}\left(  i\right)  ,S^{\prime}\right)  }$, that is, $q^{\beta_{i}-\gamma_{i}%
}t^{c\left(  r_{\beta}\left(  i\right)  ,S\right)  -c\left(  r_{\gamma}\left(
i\right)  ,S^{\prime}\right)  }=1$ at $\left(  q,t\right)  =\varpi$. By Lemma
\ref{qtab} there is an integer $v_{i}$ such that $\beta_{i}-\gamma_{i}=mv_{i}$
and $c\left(  r_{\beta}\left(  i\right)  ,S\right)  -c\left(  r_{\gamma
}\left(  i\right)  ,S^{\prime}\right)  =nv_{i}$. This argument applies to all
$i$.
\end{proof}

\begin{proposition}
\label{Mpoles1}Suppose $\left(  \beta,S\right)  \in\mathbb{N}_{0}^{N}%
\times\mathcal{Y}\left(  \tau\right)  $ and there are no $\left(  m,n\right)
$-critical pairs $\left[  \left(  \beta,S\right)  ,\left(  \gamma,S^{\prime
}\right)  \right]  $ then $M_{\beta,S}$ has no poles at $\left(  q,t\right)
=\varpi$.
\end{proposition}

\begin{proof}
The proof is essentially identical to that of Proposition \ref{Jpoles1}. There
replace $x^{\beta}\otimes S\tau\left(  r_{\alpha}\right)  $ by $q^{a}%
t^{b}x^{\beta}\otimes S\tau\left(  R_{\beta}\right)  $ (with the appropriate
prefactor $q^{a}t^{b}$), $J$ by $M$, $\zeta$ by $\widetilde{\zeta}$,
$\mathcal{U}_{i}$ by $\xi_{i}$. The formula shows that $M_{\beta,S}$ is a
polynomial, the denominators of whose coefficients are products of factors
with the form $q^{\beta_{i}}t^{b}-q^{\gamma_{i}}t^{b^{\prime}}$, and none of
these vanish at $\left(  q,t\right)  =\varpi$.
\end{proof}

This is our main result for the Macdonald polynomials.

\begin{theorem}
\label{mainMthm}Suppose $\alpha=\left(  m,0,\ldots\right)  ,S\in
\mathcal{Y}\left(  \tau\right)  $ and $\mathcal{Z}\left(  \widehat{\tau
}\right)  $ is as in Definition \ref{T_tau}. Further suppose $z\in
\mathcal{Z}\left(  \widehat{\tau}\right)  $, $n:=c\left(  1,S\right)  -z\neq0$
then $M_{\alpha,S}$ is a singular polynomial for $\left(  q,t\right)  =\varpi$.
\end{theorem}

\begin{proof}
From Proposition \ref{MDi0} $M_{\alpha,S}\mathcal{D}_{j}=0$ for $2\leq j\leq
N$ and $M_{\alpha,S}\mathcal{D}_{1}=C_{S,m}\left(  q,t\right)  M_{\widehat
{\alpha},S}$, where $\widehat{\alpha}=\left(  m-1,0,\ldots\right)  .$By
Propositions \ref{noCpair}, \ref{Mpoles1} and Lemma \ref{JMpoles2}
$M_{\alpha,S}$ and $M_{\widehat{\alpha},S}$ do not have poles at $\left(
q,t\right)  =\varpi$. Furthermore\linebreak$C_{S,m}\left(  \omega
u^{-n/g},u^{m/g}\right)  =0$ (due to the factor $1-q^{m}t^{c\left(
1,S\right)  -z}$, Proposition \ref{Cqtzeros}) and thus $M_{\alpha
,S}\mathcal{D}_{1}=0$ at $\left(  q,t\right)  =\varpi$.
\end{proof}

\subsection{Isotype of Singular Polynomials\label{Isotyp}}

The following discussion is in terms of Macdonald polynomials. It is
straightforward to deduce the analogous results for Jack polynomials. Suppose
$\sigma$ is a partition of $N$. A basis $\left\{  p_{S}:S\in\mathcal{Y}\left(
\sigma\right)  \right\}  $ of an $\mathcal{H}_{N}\left(  t\right)  $-invariant
subspace of $\mathcal{P}_{\tau}$ is called \textit{a basis of isotype}
$\sigma$ if each $p_{S}$ transforms under the action of $T_{i}$ defined in
Section \ref{RepTh} with $\sigma\left(  T_{i}\right)  $ replaced by $T_{i}$.
For example if $\mathrm{\operatorname{row}}\left(  i,S\right)
=\mathrm{\operatorname{row}}\left(  i+1,S\right)  $ then $p_{S}\left(
xs_{i}\right)  =p_{S}\left(  x\right)  $, equivalently $p_{S}T_{i}=tp_{S}$, or
if $\operatorname{col}\left(  i,S\right)  =\operatorname{col}\left(
i+1,S\right)  $ then $p_{S}T_{i}=-p_{S}$. There is a strong relation to
singular polynomials.

\begin{proposition}
: A polynomial $p\in\mathcal{P}_{\tau}$ is singular for a specific value of
$\left(  q,t\right)  =\psi$ if and only if $p\xi_{i}=p\phi_{i}$ for $1\leq
i\leq N$, evaluated at $\psi$.
\end{proposition}

\begin{proof}
Recall the Jucys-Murphy elements $\left\{  \phi_{i}\right\}  $ from
(\ref{JMurph}). By definition $p\mathcal{D}_{N}=0$ if and only if $p\xi
_{N}=p=p\phi_{N}$. Proceeding by induction suppose that $p\mathcal{D}_{j}=0$
for $i<j\leq N$ if and only if $p\xi_{j}=p\phi_{j}$ for $i<j\leq N$. Suppose
\begin{align*}
0  &  =p\mathcal{D}_{i}=\frac{1}{t}pT_{i}\mathcal{D}_{i+1}T_{i}%
\Longleftrightarrow pT_{i}\mathcal{D}_{i+1}=0\Longleftrightarrow pT_{i}%
\xi_{i+1}=pT_{i}\phi_{i+1}\\
&  \Longleftrightarrow pT_{i}\xi_{i+1}T_{i}=pT_{i}\phi_{i+1}T_{i}%
\Longleftrightarrow tp\xi_{i}=tp\phi_{i}.
\end{align*}
This completes the proof.
\end{proof}

With $M_{\alpha,S}$ and $n$ as in Theorem \ref{mainMthm} the spectral vector
$\left[  \widetilde{\zeta}_{\alpha,S}\left(  i\right)  \right]  _{i=1}%
^{N}\allowbreak=\left[  q^{m}t^{c\left(  1,S\right)  },t^{c\left(  2,S\right)
},\ldots,t^{c\left(  N,S\right)  }\right]  $. Specialized to $\left(
q,t\right)  =\varpi$ the polynomial $M_{\alpha,S}$ is singular and
$q^{m}t^{c\left(  1,S\right)  }=t^{-n+c\left(  1,S\right)  }$. Recall $n=z$
for some $z\in\mathcal{Z}\left(  \widehat{\tau}\right)  $, and $z$ determines
a cell $\left(  i_{s},\widehat{\tau}_{i_{s}}+1\right)  $. In terms of Ferrers
diagrams let $\sigma=\widehat{\tau}\cup\left(  i_{s},\widehat{\tau}_{i_{s}%
}+1\right)  $, that is $\sigma_{i_{s}}=\tau_{i_{s}}+1$. Let $S^{\prime}$
denote the RSYT formed from the cells of $\tau$ containing the numbers
$2,\ldots,N$ and the cell $\left(  i_{s},\widehat{\tau}_{i_{s}}+1\right)  $
containing $1$. Then $c\left(  i,S^{\prime}\right)  =c\left(  i,S\right)  $
for $2\leq i\leq N$ and $c\left(  1,S^{\prime}\right)  =c\left(  1,S\right)
-n$. Thus the spectral vector of $M_{\alpha,S}$ evaluated at $\left(
q,t\right)  =\varpi$ is $\left[  t^{c\left(  i,S^{\prime}\right)  }\right]
_{i=1}^{N}$. This implies that $M_{\alpha,S}$ is (a basis element) of isotype
$\sigma$. The other elements of the basis corresponding to $\mathcal{Y}\left(
\sigma\right)  $ are obtained from $M_{\alpha,S}$ by appropriate
transformations using $\left\{  T_{i}\right\}  $.

\section{Concluding Remarks}

We have shown the existence of singular vector-valued Jack and Macdonald
polynomials for the easiest possible values of the label $\alpha$, that is,
$\left(  m,0,\ldots,0\right)  $. The proofs required some differentiation
formulas and combinatorial arguments involving Young tableaux. The singular
values were found to have an elegant interpretation in terms of where another
cell can be attached to an RSYT. It may occur that a larger set of parameter
values, say $\gcd\left(  m,n\right)  >1$, or even $\frac{m}{n}\notin
\mathbb{Z}$, still leads to singular Jack polynomials but our proof techniques
do not seem to cover these. One hopes that eventually a larger class of
examples (more general labels in $\mathbb{N}_{0}^{N}$) will be found, with a
target of a complete listing as is already known for the trivial
representation $\tau=\left(  N\right)  $. It is suggestive that the isotype
$\sigma$ of the singular polynomial $M_{\alpha,S}$ is obtained by a reasonably
natural transformation of the partition $\tau$.

\end{document}